\documentclass[10pt,twoside,a4paper,reqno]{amsart}
\usepackage{bulletinUPB_3}
\usepackage{amsfonts,amsthm,amsmath,graphicx,enumerate}
\usepackage{graphics,graphicx,amssymb,eucal,amsmath}

\info{A}{72}{2}{2010}
\theoremstyle{plain}

\theoremstyle{definition}

\makeatletter \@addtoreset{equation}{section} \makeatother

\def\beq#1{\begin{equation}\label{#1}}
\def\eeq{\end{equation}}

\title{\hspace{0.8cm} {Viscosity solutions of divergence type PDEs} \newline \hspace{-1.3cm}{associated to multitime hybrid games}}

\author{Constantin {\sc Udri\c ste}}
\address{$^{1}$Professor, Department of Mathematics and Informatics, Faculty of Applied
Sciences, University {\sc Politehnica} of Bucharest, Splaiul
Independentei 313,RO-060042, Bucharest, Romania.
   E-mail: {\tt udriste@mathem.pub.ro}}
\author{Ionel {\sc \c Tevy}}
\address{$^{2}$Professor, Department of Mathematics and Informatics, Faculty of Applied
Sciences, University {\sc Politehnica} of Bucharest, Splaiul
Independentei 313,RO-060042, Bucharest, Romania.
   E-mail: {\tt vascatevy@yahoo.fr}}
\author{Elena-Laura {\sc Otob\^{i}cu}}
\address{$^{3}$PhD student, Department of Mathematics and Informatics, Faculty of Applied
Sciences, University {\sc Politehnica} of Bucharest, Splaiul
Independentei 313,RO-060042, Bucharest, Romania.
E-mail: {\tt laura.otobicu@gmail.com}}

\begin{document}
\pagestyle{headings} \maketitle

\begin{abstract}
{\it Our original results are associated to a multitime hybrid game, with two equips of players,
based on a multiple integral functional
and an $m$-flow as constraint. The aim of this paper is three-fold:
(i) to define the multitime lower or upper value function;
(ii) to build Divergence type PDEs ;
(iii) to give viscosity solutions of previous PDEs.}
\end{abstract}


\begin{MSC2010}
49L20, 91A23, 49L25.
\end{MSC2010}

\section{Multitime lower or upper value function}
All variables and functions
must satisfy suitable conditions (for example, see \cite{[11]}).
We refer to a {\it multitime hybrid differential game, with two equips of players}, whose {\it Bolza payoff} is
the sum between a multiple integral (volume) and a function of the final
event (the terminal cost) and whose evolution PDE is an $m$-{\it flow}.
More specifically, this paper refers to the following optimal control problem:

{\it Find
$$\min_{v(\cdot)\in V}\max_{u(\cdot)\in U}I(u(\cdot),v(\cdot))=\int_{\Omega_{0T}} L (s,x(s),u(s),v(s))ds + g(x(T)),$$
subject to the Cauchy problem
$$\frac{\partial x^i}{\partial s^\alpha}(s)=X^i_\alpha(s,x(s),u(s),v(s)),\,\,
x(0)=x_0,\, s\in \Omega_{0T}\subset \mathbb{R}_+^m,\, x\in \mathbb{R}^n,$$}
where $i=1,...,n$; $\alpha =1,...,m$; $u=(u^a)$, $a=1,...,p$, $v=(v^b)$, $b=1,...,q$ are the controls;
$ds=ds^1\wedge...\wedge ds^m$ is the volume element.

We vary the starting multitime and the initial point. We obtain a larger family of
similar multitime problems based on the functional
$$I_{t,x}(u(\cdot),v(\cdot))=\int_{\Omega_{tT}} L (s,x(s),u(s),v(s))ds+g(x(T))$$
and the multitime evolution constraint (Cauchy problem for first order PDEs system)
$$\frac{\partial x^i}{\partial s^\alpha}(s)=X^i_\alpha(s,x(s),u(s),v(s)),\,\,
x(t)=x,\, s\in \Omega_{tT}\subset \mathbb{R}_+^m,\, x\in \mathbb{R}^n.$$

\begin{remark} Sometimes, the existence of multitime totally min-max optimal controls and
sheets can be seen without any optimality arguments. For example, the problem
$$V(x)=\min_u \,\max_v\,\int_{\Omega_{tT}}((x(s)+u(s))^2-v(s)^2)ds,\,\,\frac{\partial x}{\partial s^\alpha}(s)=X_\alpha(s,x(s),u(s),v(s)),$$
with $s\in \Omega_{tT}$, $t\in \Omega_{0T}$,\,$x(0)=x_0$,
has a global min-max solution $V(x)=0$ for $u=-x,v=0$, and all $t,T,x_0$.
The set of those sheets is obviously a totally optimal field of sheets corresponding
to given $X_\alpha$ for which solutions of $\frac{\partial x}{\partial s^\alpha}(s)=X_\alpha(s,x(s),-x(s),0)$
exist over $[0,T]$ for any $x_0$.
\end{remark}

\begin{remark} Let us consider the problem: find $u(t)$ and $F(t)$ such that
$$W(F)=\min_u\int_{\Omega}\sum_{\alpha =1}^m(t^\alpha +u_\alpha(t))^2dt,\,\,\,u\nabla F=0,\,t\in \mathbb{R}^m, \,\Omega\subset \mathbb{R}^m.$$
This problem has a global optimal solution
$$W(F)=0,\, \hbox{for}\,\,\,u_\alpha(t) = - t^\alpha \not=0, \,F(t)= \varphi\left(\frac{t^1}{t^2},\cdots, \frac{t^\alpha}{t^{\alpha+1}},\cdots,
\frac{t^{m-1}}{t^m}, \frac{t^m}{t^1}\right),$$
with arbitrary function $\varphi$, any subset $\Omega$ and $m\geq 2$.
\end{remark}

\begin{definition} Let $\Psi$ and $\Phi$ be suitable strategies of the two equips of players.

(i) The function
$$ m(t,x)=\min_{\Psi\in \mathcal{B}} \max_{u(\cdot)\in \mathcal{U}} I_{t,x}[u(\cdot),\Psi[u](\cdot)]$$
is called the multitime lower value function.

(ii) The function
$$M(t,x)=\max_{\Phi\in \mathcal{A}} \min_{v(\cdot)\in \mathcal{V}} I_{t,x}[\Phi[v](\cdot),[v](\cdot)]$$
is called the multitime upper value function.
\end{definition}

The most important ingredient in our theory is the idea of {\it generating vector field}.
This mathematical ingredient allow the introduction of PDEs of divergence type (see \cite{[11]}).

\begin{definition} Let $D_{\alpha}$ be the total derivative and $c_{hyp}$, $C_{hyp}$ be hyperbolic constants.

(i) The vector field $(\mathrm{m}^{\alpha}(t,x))$ is called
the {\it generating lower vector field} of the lower value function $m(t,x),$ if
$$m(T,x(T))=c_{hyp}+m(t,x(t))+\int_{\Omega_{tT}}D_{\alpha}\mathrm{m}^{\alpha}(s,x(s))ds.$$

(ii) The vector field $(\mathrm{M}^{\alpha}(t,x))$ is called the {\it generating upper vector field}
of the upper value function $M(t,x),$ if
$$M(T,x(T))=C_{hyp}+M(t,x(t))+\int_{\Omega_{tT}}D_{\alpha}\mathrm{M}^{\alpha}(s,x(s))ds.$$
\end{definition}

The papers \cite{[1]}-\cite{[4]}, \cite{[12]} refer to viscosity solutions of Hamilton-Jacobi-Isaacs equations.
Our papers \cite{[5]}-\cite{[11]} are listed for understanding the multitime optimal control
and our recent results which led to the present work.

\section{Viscosity solutions of divergence type PDEs}

Let us recall some PDEs that admit viscosity solutions:
(i) the {\it Eikonal Equation}: $|Du| = f(x)$, which is related to geometric optics (rays);
(ii) {\it (stationary) Hamilton-Jacobi equation}: $H(x, u, Du) = 0,\, \Omega \subset \mathbb{R}^n$,
where $H : \Omega \times \mathbb{R} \times \mathbb{R}^n \to \mathbb{R}$ is called Hamiltonian and is continuous
and in general convex in $p$ (i.e. in the gradient-variable); the eikonal
equation is in particular a (stationary) Hamilton-Jacobi equation;
(iii) {\it (single time evolution) Hamilton-Jacobi equation}: $u_t + H(x, u, Du) = 0$, $\mathbb{R}^n \times (0, \infty)$;
(iv) the {\it single time Hamilton-Jacobi-Bellman equation}, based on the Hamiltonian
$$H(x,p)=\sup_{a\in A}\{-f(x,a)\cdot p - l(x,a);$$
this is a particular Hamilton-Jacobi equation which is very important in
single-time control theory and economics;
(v) in {\it Differential Games}, with two equips of players, the (lower) value function
$$v(t,x) = \inf_{\alpha \in \mathcal{A}[\mathcal{B}]}\sup_{\beta\in \mathcal{B}}J(t,x,\alpha, \beta)$$
solves the Hamilton-Jacobi-Bellman PDE
$$ u_t + H(x, Du) = 0,\, x \in \mathbb{R}^n,$$
where
$$H(t,x, p) := \min_{b\in B}\max_{a\in A}\{-f(t,x,a,b)\cdot p - l(t,x,a,b)\}.$$
(vi) there exists {\it multitime Hamilton-Jacobi PDEs systems} which admit no
viscosity solution, in the non-convex setting, even when the Hamiltonians
are in involution \cite{[12]}.

Our aim is to introduce multitime divergence type PDEs that admit viscosity solutions.
Viscosity solutions need not be differentiable anywhere and thus are not sensitive to
the classical problem of the crossing of characteristics.

The generating lower and upper vector fields, denoted by $(\mathrm{m}^\alpha)$ and $(\mathrm{M}^\alpha)$, define the relations
$$m(t)-m(t+h)=-c_{hyp}-\int_{\Omega_{tt+h}} D_\alpha \mathrm{m}^\alpha\, ds.$$
$$M(t)-M(t+h)=-C_{hyp}-\int_{\Omega_{tt+h}} D_\alpha \mathrm{M}^\alpha\, ds.$$

The key original idea is that the generating
upper vector field or the generating lower vector field
are solutions of divergence type PDEs, defined in the next Theorem.
Our PDEs contain some implicit assumptions, and are valid under certain conditions
which are defined and analyzed for multitime hybrid differential games.

\begin{theorem}
(i) The generating upper vector field $(\mathrm{M}^\alpha(t,x))$ is the viscosity solution of
the multitime upper divergence type PDE
$$\frac{\partial \mathrm{M}^{\alpha}}{\partial t^\alpha}(t,x)+\min_{v\in V} \max_{u\in U} \left\lbrace  \frac{\partial \mathrm{M}^{\alpha}}{\partial x^i}(t,x) X_\alpha^i(t,x,u,v)+L(t,x,u,v)\right\rbrace =0,$$
which satisfies the terminal condition $\mathrm{M}^\alpha(T,x)=g^\alpha(x).$

(ii) The generating lower vector field $(\mathrm{m}^\alpha(t,x))$ is the viscosity solution of
the multitime lower divergence type PDE
$$\frac{\partial \mathrm{m}^\alpha}{\partial t^\alpha}+\max_{u \in U} \min_{v \in V} \left\lbrace  \frac{\partial \mathrm{m}^\alpha}{\partial x^i}(t,x) X_\alpha^i(t,x,u,v)+L(t,x,u,v)\right\rbrace =0,$$
which satisfies the terminal condition $\mathrm{m}^\alpha(T,x)=g^\alpha(x).$
\end{theorem}

\begin{proof}
To simplify the divergence type PDEs, we introduce
the so-called upper and lower Hamiltonian defined respectively by
$$H^+(t,x,p)=\min_{v\in \mathcal{V}} \max_{u \in \mathcal{U}}\lbrace p_i^\alpha(t) X_\alpha^i(t,x,u,v)+L(t,x,u,v)\rbrace,$$
$$H^-(t,x,p)=\max_{u\in \mathcal{U}} \min_{v\in \mathcal{V}}\lbrace p_i^\alpha(t) X_\alpha^i(t,x,u,v)+L(t,x,u,v)\rbrace.$$

We prove only the first statement. For $s\in \Omega_{tt+h},$ we use the Cauchy problem
$$\frac{\partial x^i}{\partial s^\alpha}(s)=X^i_\alpha(s,x(s),u(s),v(s)),\,\,
x(t)=x, \,s\in \Omega_{tt+h}\subset \mathbb{R}_+^m,\, x\in \mathbb{R}^n$$
and the cost functional (volume)
$$I_{t,x}(u(\cdot),v(\cdot))=\int_{\Omega_{tt+h}} L (s,x(s),u(s),v(s))\,ds.$$
For $s\in \Omega_{tT}\setminus \Omega_{tt+h}$, the cost is $M(t+h,x(t+h))$.
Consequently, $$I_{t,x}(u(\cdot),v(\cdot))=\int_{\Omega_{tt+h}} L(s,x(s),u(s),v(s))\,ds+M(t+h,x(t+h)),$$
with $M(t,x)\geq M(t+h,x(t+h))$, because $M(t,x)$ is the greatest cost.
Thus we have the multitime dynamic programming optimality condition
$$
M(t,x)=\max_{\Phi\in \mathcal{A}(t)}\min_{v\in \mathcal{V}(t)}\bigg\{\int_{\Omega_{tt+h}} L (s,x(s),\Phi [v](s),v(s))\,ds +M(t+h,x(t+h))\bigg\}.
$$

Let $(\mathrm{w}^\alpha)\in C^1(\Omega_{0T}\times \mathbb{R}^n)$ be
a generating vector field.
We analyse two cases:

\medskip
{\bf Case 1} Suppose $\mathrm{M}^\alpha-\mathrm{w}^\alpha$ attains a local maximum at $(t,x)\in \Omega_{0T}\times \mathbb{R}^n.$
To prove the inequality
$$\frac{\partial \mathrm{w}^\alpha}{\partial t^\alpha}(t,x)+H^+\left( t,x,\frac{\partial \mathrm{w}}{\partial x^i}(t,x)\right) \geq 0,\eqno(1)$$
it is enough to prove that the relation
$$
\frac{\partial \mathrm{w}^\alpha}{\partial t^\alpha}(t,x)+H^+\left( t,x,\frac{\partial \mathrm{w}}{\partial x^i}(t,x)\right) \leq -\theta <0,
$$
is false, for some constant $\theta>0.$

We use the {\it Fundamental Lemma} in the next Section. For a sufficiently small $||h|| >0$,
all ${w}\in \mathcal{A}(t),$  for $v\in V(t),$ we obtain that the relation
$$
\int_{\Omega_{tt+h}} \bigg( L(s,x(s),\Phi[v](s),v(s)) +\frac{\partial {w}^\alpha}{\partial x^i} X_\alpha^i(s,x(s),\Phi[v](s),v(s))+\frac{\partial {w}^\alpha}{\partial s^\alpha}\bigg)ds\leq -\frac{\theta\, \hbox{vol}\,(h)}{2}
$$
holds for $v\in \mathcal{V}.$ Thus
$$
\max_{\Phi\in \mathcal{A}(t)}\min_{v\in \mathcal{V}(t)} \bigg\{  \int_{\Omega_{tt+h}} \bigg(L(s,x(s),\Phi[v](s),v(s)) +\frac{\partial {w}^\alpha}{\partial x^i} X_\alpha^i(s,x(s),\Phi[v](s),v(s))+\frac{\partial {w}^\alpha}{\partial s^\alpha}\bigg)ds\bigg\}
$$
$$
\leq -\frac{\theta\, \hbox{vol}\,(h)}{2},\eqno(2)
$$
with $x(\cdot)$ solution of the previous Cauchy problem.

Because $\mathrm{M}^\alpha -\mathrm{w}^\alpha$ has a local maximum at $(t,x),$ we have
$$\mathrm{w}^\alpha(t+h,x(t+h))-\mathrm{w}^\alpha(t,x)\geq \mathrm{M}^\alpha(t+h,x(t+h))-\mathrm{M}^\alpha(t,x).$$
Dividing by $h^\alpha$, taking the limit for $h^\alpha \to 0$, summing after $\alpha$, we find the inequality
$D_\alpha \mathrm{w}^\alpha \geq D_\alpha \mathrm{M}^\alpha$
and finally
$$M(t,x)-M(t+h,x(t+h))\geq w(t,x)-w(t+h,x(t+h)).$$

The multitime dynamic programming optimality condition and the local maximum definition give us
$$M(t,x)-M(t+h,x(t+h))=\max_{\Phi\in \mathcal{A}(t)}\min_{v\in \mathcal{V}(t)}\left\lbrace  \int_{\Omega_{tt+h}} L(s,x(s),\Phi [v](s),v(s))ds\right\rbrace. $$

Consequently, we have
$$\max_{\Phi\in \mathcal{A}(t)}\min_{v\in \mathcal{V}(t)}\left\lbrace  \int_{\Omega_{tt+h}} L (s,x(s),\Phi [v](s),v(s))ds\right\rbrace \geq {w}(t,x)-{w}(t+h,x(t+h))$$
or
$$\max_{\Phi\in \mathcal{A}(t)}\min_{v\in \mathcal{V}(t)}\left\lbrace  \int_{\Omega_{tt+h}} L(s,x(s),\Phi [v](s),v(s))ds\right\rbrace +{w}(t+h,x(t+h))-{w}(t,x)\geq 0.$$

By the definition of the generating vector field, the $m$-dimensional hyperbolic difference is
$$
{w}(t+h,x(t+h))-{w}(t,x)=\int_{\Omega_{tt+h}} D_\alpha \mathrm{w}^\alpha ds
 $$
 $$
 =\int_{\Omega_{tt+h}} \bigg(\frac{\partial {w}^\alpha}{\partial x^i} X_\alpha^i(s,x(s),\Phi[v](s),v(s))+\frac{\partial {w}^\alpha}{\partial s^\alpha}\bigg)ds
$$
 and thus the contradiction with $(2)$ arises.

\medskip
{\bf Case 2} Suppose $\mathrm{M}^\alpha-\mathrm{w}^\alpha$ attains a local minimum at $(t,x)\in \Omega_{0T}\times \mathbb{R}^n.$
Analogously to the previous case, we shall prove the relation
$$
\frac{\partial {w}^\alpha}{\partial t^\alpha}(t,x)+H^+(t,x,\frac{\partial {w}}{\partial x^i}(t,x))\leq 0\eqno(3)
$$
by supposing the contrary
$$
\frac{\partial {w}^\alpha}{\partial t^\alpha}(t,x)+H^+(t,x,\frac{\partial {w}}{\partial x^i}(t,x))\geq \theta >0,
$$
for some constant $\theta>0.$

According to {\it Fundamental Lemma} (see next Section), for each sufficiently
small $||h|| >0$ and all ${w}\in \mathcal{A}(t),$ the foregoing results imply that the relation
$$
\int_{\Omega_{tt+h}} \left(L(s,x(s),\Phi[v](s),v(s))+\frac{\partial {w}^\alpha}{\partial x^i} X_\alpha^i(s,x(s),\Phi[v](s),v(s)) +\frac{\partial {w}^\alpha}{\partial s^\alpha}\right)ds\geq\frac{\theta\, \hbox{vol}(h)}{2}
$$
holds for $v\in V(t)$. It follows that the relation
$$
\max_{\Phi\in \mathcal{A}(t)}\min_{v\in \mathcal{V}(t)} \bigg\{\int_{\Omega_{tt+h}} \left(L(s,x(s),\Phi[v](s),v(s))+\frac{\partial {w}^\alpha}{\partial x^i} X_\alpha^i(s,x(s),\Phi[v](s),v(s))\right.$$
$$
\hspace{5cm}\left.+\frac{\partial {w}^\alpha}{\partial s^\alpha}\right)ds\bigg\}\geq\frac{\theta\, \hbox{vol}(h)}{2}\eqno(4)
$$
is true.

Because $\mathrm{M}^\alpha-\mathrm{w}^\alpha$ has a minimum at $(t,x),$ we have the inequality
$$
\mathrm{M}^\alpha(t,x)-\mathrm{w}^\alpha(t,x)\leq \mathrm{M}^\alpha(t+h,x(t+h))-\mathrm{w}^\alpha(t+h,x(t+h))
$$
and so
$D_\alpha \mathrm{w}^\alpha \leq D_\alpha \mathrm{M}^\alpha.$
 By the local minimum definition and by the multitime dynamic programming optimality condition, we can state the relation
$$M(t,x)-M(t+h,x(t+h))=\max_{\Phi\in \mathcal{A}(t)}\min_{v\in \mathcal{V}(t)}\left\lbrace  \int_{\Omega_{tt+h}} L (s,x(s),\Phi [v](s),v(s))ds\right\rbrace. $$
Consequently, we obtain
$$
\max_{\Phi\in \mathcal{A}(t)}\min_{v\in \mathcal{V}(t)}\bigg\{ \int_{\Omega_{tt+h}} L (s,x(s),\Phi [v](s),v(s))ds\bigg\} +{w}(t+h,x(t+h))-{w}(t,x)\leq 0.
$$

On the other hand, we know that
$$
{w}(t+h,x(t+h))-{w}(t,x)=\int_{\Omega_{tt+h}} D_\alpha{w}^\alpha ds
$$
$$
=\int_{\Omega_{tt+h}} \bigg(\frac{\partial {w}^\alpha}{\partial x^i} X_\alpha^i(s,x(s),\Phi[v](s),v(s))+\frac{\partial {w}^\alpha}{\partial s^\alpha}\bigg)ds.
$$
The last two relations contradicts the relation $(4)$ and thus the relation $(3)$ must be true.
\end{proof}

\section{Fundamental contradict Lemma}

It is preferable to put the Lemma in this Section to receive notations from the previous Section.
\begin{lemma}\label{L-1}
Let ${w}\in C^1(\Omega_{0T}\times \mathbb{R}^n)$ and the associated generating vector field
$(\mathrm{w}^\alpha)$.

(i) If $\mathrm{M}^\alpha-\mathrm{w}^\alpha$ attains a local maximum at $(t_0,x_0)\in \Omega_{0T}\times \mathbb{R}^n$,
for each $\alpha,$ and
$$\frac{\partial\mathrm{w}^\alpha}{\partial t^\alpha}(t_0,x_0)+ H^+\left(t_0,x_0,\frac{\partial\mathrm{w}}{\partial x^i}(t_0,x_0)\right)\leq - \theta <0,$$
then for all sufficiently small $||h|| >0$, there exists a control $v=(v_\alpha)\in {\mathcal V}(t_0)$ such that the relation $(2)$ holds for all strategies $\Phi\in {\mathcal A}(t_0)$.

(ii) If $\mathrm{M}^\alpha-\mathrm{w}^\alpha$ attains a
local minimum at $(t_0,x_0)\in \Omega_{0T}\times \mathbb{R}^n$, for each $\alpha,$ and
$$\frac{\partial\mathrm{w}^\alpha}{\partial t^\alpha}(t_0,x_0)+ H^+\left(t_0,x_0,\frac{\partial\mathrm{w}}{\partial x^i}(t_0,x_0)\right)\geq \theta >0,$$
then for all sufficiently small $||h|| >0$, there exists a control $u=(u_\alpha)\in {\mathcal U}(t_0)$ such that the relation $(4)$ holds for all strategies $\Psi\in {\mathcal B}(t_0)$.
\end{lemma}

\begin{proof} The basic object is the $m$-form (identified with one component function, see multiple integral)
$$\Lambda= L(s,x(s),\Phi[v](s),v(s))+\frac{\partial \mathrm{w}^\alpha}{\partial x^i} X_\alpha^i(s,x(s),\Phi[v](s),v(s))+\frac{\partial \mathrm{w}^\alpha}{\partial t^\alpha}.
$$

(i) By hypothesis
$$\min_{v\in {\mathcal V}}\,\, \max_{u\in {\mathcal U}} \,\Lambda (t_0,x_0,u,v)\leq -\theta <0.$$
Consequently there exists some control $v^*\in {\mathcal V}$ such that
$$\max_{u\in {\mathcal U}} \, \Lambda (t_0,x_0,u,v^*)\leq -\theta.$$
By the uniform continuity of the $m$-form $\Lambda,$ we have
$$\max_{u\in {\mathcal U}} \, \Lambda (t_0,x(s),u,v^*)\leq - \frac{1}{2}\,\theta$$
provided $s\in \Omega_{t_0t_0+h}$, for any small $||h|| >0$, and $x(\cdot)$ is solution of PDE on $\Omega_{t_0t_0+h}$,
for any $u(\cdot),v(\cdot)$, with initial condition $x(t_0)=x_0$. For the control $v(\cdot)=v^*$ and for any strategy $\Phi\in {\mathcal A}(t_0)$, we find
$$
L(s,x(s),\Phi[v](s),v(s))+\frac{\partial \mathrm{w}^\alpha}{\partial x^i} X_\alpha^i(s,x(s),\Phi[v](s),v(s))+\frac{\partial \mathrm{w}^\alpha}{\partial t^\alpha}\leq \frac{1}{2}\, \theta,
$$
for $s\in \Omega{t_0t_0+h}$. Taking the hyperbolic primitive integral on $\Omega_{t_0t_0+h}$, we obtain the relation $(2)$.

(ii) The inequality in the Lemma reads
$$\min_{v\in {\mathcal V}}\,\, \max_{u\in {\mathcal U}} \,\Lambda (t_0,x_0,u,v)\geq \theta >0.$$
Consequently, for each control $v\in {\mathcal V}$ there exists a control $u=u(v)\in U$ such that
$$\Lambda (t_0,x_0,u,v)\geq \theta.$$
The uniform continuity of the $m$-form $\Lambda$ implies
$$\Lambda (t_0,x_0,u,\xi)\geq \frac{3}{4}\,\theta,\,\,\forall \xi\in B(v,r)\cap V\, \hbox{and some}\,\, r=r(v)>0.$$
Due to compactness of ${\mathcal V}$, there exists finitely many distinct points
$$v_1,...,v_n \in {\mathcal V}; \,u_1,..., u_n\in {\mathcal U}$$
and the numbers $r_1,...,r_n >0$ such that ${\mathcal V}\subset \bigcup_{i=1}^{n} B(v_i,r_i)$ and
$$\Lambda (t_0,x_0,u_i,\xi)\geq \frac{3}{4}\,\theta,\,\,\forall \xi\in B(v_i,r_i).$$

Define
$$\psi:{\mathcal V}\to {\mathcal U},\, \psi(v)=u_k\,\, \hbox{if}\,\, v\in B(u_k,r_k)\setminus \bigcup_{i=1}^{k-1} B(u_i,r_i),\,k=\overline{1,n}.$$
In this way, we have the inequality
$$\Lambda (t_0,x_0,\psi(v),v)\geq \frac{3}{4}\,\theta, \,\,\forall v\in {\mathcal V}.$$
Again, the uniform continuity of the $m$-form $\Lambda$ and a sufficiently small $||h|| >0$ give
$$\Lambda (s,x(s),\psi(v),v)\geq \frac{1}{2}\,\theta, \forall v\in {\mathcal V},\, s\in \Omega_{t_0t_0+h},$$
and any solution $x(\cdot)$ of PDE on $\Omega_{t_0t_0+h}$, for any $u(\cdot), v(\cdot)$ and with initial condition $x(t_0)=x_0$.
Now define the strategy
$$\Psi\in {\mathcal B}(t_0),\,\,\Psi[v](s)=\psi(v(s)),\, \forall v\in {\mathcal V}(t_0),\,s\in \Omega_{t_0t_0+h}.$$
Finally, we have the inequality
$$\Lambda (s,x(s),\Psi[v)](s),v(s))\geq \frac{1}{2}\,\theta, \forall \,s\in \Omega_{t_0t_0+h},$$
and taking the hyperbolic primitive integral on $\Omega_{t_0t_0+h}$, we find the result in Lemma.
\end{proof}

\section{Conclusions}

Our problem of multitime hybrid differential games requires some
original ideas issued from our research in multitime dynamic programming:
generating vector field, divergence type PDEs and their viscosity solutions.
Also, to formulate and prove such Theorems we need a proper geometric language.
This research shows how the general theory of multitime dynamic programming
can be applied to special problems, here multitime optimal games.

{\bf Acknowledgements} The results of this work were presented at
the IX-th International Conference of Differential Geometry and Dynamical Systems
(DGDS-2015), 8 - 11 October 2015, Bucharest, Romania.

Partially supported by University Politehnica of Bucharest and by Academy of Romanian Scientists.

%
%

\end{document}